\documentclass[12pt]{article}

\usepackage{mathrsfs}

\usepackage{amsmath,amsthm,amssymb}
\usepackage{graphicx}

\newtheorem{lemma}{\bf Lemma}

\newtheorem{theorem}{\bf Theorem}
\newtheorem{corollary}{\bf Corollary}
\newtheorem{remark}{\bf Remark}

\newtheorem{proposition}{\bf Proposition}

\newenvironment{Proof of lemma}{\noindent{\it Proof of Lemma:}}{\hfill$\Box$}
\newenvironment{Proof of theorem}{\noindent{\bf Proof of Theorem}}{\hfill$\Box$}
\newenvironment{Proof of exercise}{\noindent{\it Proof of Exercise:}}{\hfill$\Box$}
\newenvironment{Acknowledgements}{\noindent{\bf Acknowledgements}}

\begin{document}

\title{Central Limit Theorems for Cavity and Local Fields of the Sherrington-Kirkpatrick Model}
\author{Wei-Kuo Chen\footnote{Department of Mathematics, University of California at Irvine, 340 Rowland Hall, Irvine, CA 92697-3875, USA, email: weikuoc@uci.edu}}

\maketitle

\begin{abstract}
One of the remarkable applications of the cavity method is to
prove the Thouless-Anderson-Palmer (TAP) system of equations in
the high temperature regime of the Sherrington-Kirkpatrick (SK)
model. This naturally leads us to the important study of the limit
laws for cavity and local fields. The first quantitative results
for both fields based on Stein's method were studied by Chatterjee
\cite{Chatt10}.
Although Stein's method provides us an efficient
approach for obtaining the limiting distributions, the nature of
this method restricts the derivation of optimal and general
results. In this paper, our study based on Gaussian interpolation
obtains the CLT for the cavity field. With the help of this result,
we conclude the CLT for local fields. In both cases, more refined moment estimates are given.
\end{abstract}


\section{Introduction and Main Results}
\subsection{The Sherrington-Kirkpatrick Model and TAP Equations}
Let $N$ be a positive integer. Consider the space of configurations $\Sigma_N=\left\{-1,+1\right\}^N.$ The elements $
\sigma=(\sigma_1,\ldots,\sigma_N)\in \Sigma_N$ are called {\it spin configurations}
and $\sigma_i$'s are called {\it spins}.
Suppose that $\mathbf{g}=\left\{g_{ij}\right\}_{1\leq i<j\leq N}$ are i.i.d. standard Gaussian random variables
with $g_{ij}=g_{ji},$ which is called the {\it disorder}. For a given realization of $\mathbf{g}$, we define the Hamiltonian $H_N$,
 with inverse temperature $\beta>0$ and external field $h\in\Bbb{R}$, by
\begin{align}\label{Intro:eq0}
-H_N(\sigma)=\frac{\beta}{\sqrt{N}}\sum_{i<j\leq N}g_{ij}\sigma_i\sigma_j+h\sum_{i\leq N}\sigma_i,
\end{align}
for $\sigma\in\Sigma_N$
and we then define the {\it Gibbs measure} $G_N$ on $\Sigma_N$ by
$$
G_N(\left\{\sigma\right\})=\frac{\exp\left(-H_N(\sigma)\right)}{Z_N},
$$
where $Z_N$ is the normalizing constant, called the {\it partition function}. The model we just defined here is
the so-called Sherrington-Kirkpatrick (SK) model \cite{SK75}
and our study in this paper will concentrate on the high temperature region only.

We use $\sigma^1,\sigma^2,$ etc. to denote configurations chosen independently from the Gibbs measure (with the same given disorder).
These are also called {\it replicas} in physics.
Given a function $f$ on $\Sigma_N^n=(\Sigma_N)^n$, the {\it quenched} average of $f$ on the product space $(\Sigma_N^n,G_N^{\otimes n})$ is defined as
$$
\left<f\right>=\sum_{\sigma^1,\ldots,\sigma^n}f(\sigma^1,\ldots,\sigma^n)G_N(\left\{\sigma^1\right\})\cdots G_N(\left\{\sigma^n\right\}).
$$

One of the main approaches to studying the SK model in the high temperature phase involves
the {\it overlap} of the configurations $\sigma^1,\sigma^2,$
$$
R_{1,2}=\frac{1}{N}\sum_{j\leq N}\sigma_j^1\sigma_j^2.
$$
It turns out that in the limit, as $N$ tends to infinity, this quantity will converge a.s. to a constant $q$, which is the unique solution to
\begin{align}\label{Intro:eq2}
q=E\tanh^2(\beta z\sqrt{q}+h),
\end{align}
where $z$ is a standard Gaussian random variable. More precisely,
concluding from Theorem 1.4.1 in Talagrand \cite{Talag10}, for fixed $\beta_0<1/2,$ we have
\begin{align}\label{Intro:eq1}
E\left<(R_{1,2}-q)^{2k}\right>\leq \frac{K}{N^k},
\end{align}
for every $\beta\leq \beta_0$ and $h$, where $K$ is a constant depending on $k$ and $\beta_0$ only.
By using replicas, many quantities that will be used in our study
can be controlled through $(\ref{Intro:eq1})$ (see Section 1.10 of \cite{Talag10}  for details): Let
$\dot{\sigma}_j^i=\sigma_j^i-\left<\sigma_j\right>.$ If we define
\begin{align}\label{Intro:eq4}
T_{i}=\frac{1}{N}\sum_{j\leq N}\dot{\sigma}_j^i\left<\sigma_j\right>,\quad T_{i,i}=\frac{1}{N}\sum_{j\leq N}(\dot{\sigma}^i_j)^2-(1-q),\quad T_{i,j}=\frac{1}{N}
\sum_{j'\leq N}\dot{\sigma}_{j'}^i\dot{\sigma}_{j'}^{j},\,\,i\neq j,
\end{align}
then, for fixed $\beta_0<1/2$ and $k\in\Bbb{N},$
there exists some $K$ depending on $\beta_0$ and $k$ only
such that for any $\beta\leq \beta_0$ and $h,$
\begin{align}\label{Intro:eq3}
\max_{1\leq i,j\leq N,i\neq j}
\left\{E\left<|T_i|^{2k}\right>,\,\,E\left<|T_{i,i}|^{2k}\right>,\,\,E\left<|T_{i,j}|^{2k}\right>\right\}\leq \frac{K}{N^k}.
\end{align}

Since each spin takes only two values, the Gibbs measure can be completely determined by the quenched averages
$\left<\sigma_1\right>,\ldots,\left<\sigma_N\right>.$
This observation provides us another main approach to studying SK model in the high temperature regime via the Thouless-Anderson-Palmer
(TAP) system of equations as outlined in \cite{Thoul77}:
$$
\left<\sigma_i\right>\thickapprox \tanh\left(\frac{\beta}{\sqrt{N}}\sum_{j\leq N,j\neq i}g_{ij}\left<\sigma_j\right>+h-\beta^2(1-q)\left<\sigma_i\right>\right),
\quad 1\leq i\leq N.
$$
Here $\thickapprox$ means that two quantities are approximately equal with high probability. The first rigorous proof of the validity of TAP equations
in the high temperature phase based on the remarkable formulation of the
{\it cavity} method
was established by Theorem 2.4.20
in Talagrand's book \cite{Talag03}. Later in the forthcoming edition \cite{Talag10} of \cite{Talag03}, Corollary 1.7.8 implies
that all $N$ equations hold simultaneously with high probability.

In view of the cavity method, the idea is to reduce the $N$-spin system to a smaller system by removing a spin. More precisely the procedure
can be described as follows. Recall formula $(\ref{Intro:eq0})$ for the Hamiltonian $H_N$ on $\Sigma_N$
with inverse temperature $\beta$ and external field $h.$ Let $1\leq i\leq N$ be a fixed site and write
$$
-H_N(\sigma)=-H_{N-1}(\rho)+\sigma_i\left(l_i+h\right),\quad\sigma=(\sigma_1,\ldots,\sigma_N),
$$
where $H_{N-1}$ is the Hamiltonian for the $(N-1)$-spin system defined by
\begin{align}\label{Intro:eq6}
\begin{split}
-H_{N-1}(\rho)=&\frac{\beta_-}{\sqrt{N-1}}\sum_{i'<j'\leq N,i',j'\neq i}g_{i'j'}\sigma_{i'}\sigma_{j'}+h\sum_{i'\leq N,i'\neq i}\sigma_{i'},
\end{split}
\end{align}
for $\rho=(\sigma_1,\ldots,\sigma_{i-1},\sigma_{i+1},\ldots,\sigma_{N})\in\Sigma_{N-1}$ and $\beta_-=\sqrt{(N-1)/N}\beta,$
and $l_i$ is defined as
\begin{align}\label{Intro:eq7}
l_i=&\frac{1}{\sqrt{N}}\sum_{j\leq N,j\neq i}g_{ij}\sigma_j.
\end{align}
Here, for convenience, the dependence of $H_{N-1}$ on the $i$-th spin remains implicit and we
use $\left<\cdot\right>_-$ to denote the quenched average induced by $H_{N-1}$.
Then the following fundamental identity holds
\begin{align}\label{Intro:eq5}
\left<\sigma_i\right>=\frac{\left<\sinh (\beta l_i+h)\right>_-}{\left<\cosh
(\beta l_i+h)\right>_-}=\left<\tanh(\beta l_i+h)\right>.
\end{align}
Under $\left<\cdot\right>_-$, we call $l_i$ the {\it cavity} field (slightly different from the cavity field defined in Section \ref{LLCF})
and under $\left<\cdot\right>$, we call $l_i$ the {\it local} field.
Therefore, to prove the TAP equations, we are led to the important study of the limit laws for cavity and local fields.
The key observation to establish the limit law for the cavity field is that
from the definition of $l_i,$ $\left\{g_{ij}\right\}_{j\leq N,j\neq i}$ is independent of the randomness of $\left<\cdot\right>_-,$
which motivates our approach by using Gaussian interpolation on the cavity field.
Consequently, we then deduce the limit law for the local field. In both cases, the quantitative results
for the moment estimates are given and will be stated in the following section.


\subsection{Main Results}\label{MainResults}

In the rest of the paper $\phi_{\mu,\sigma^2}$ stands for the Gaussian density with mean $\mu$ and variance $\sigma^2.$
Suppose $U:\Bbb{R}^d\rightarrow\Bbb{R}$ is continuous. We say
that $U$ is of moderate growth if
$\lim_{\|\mathbf{x}\|\rightarrow\infty}U(\mathbf{x})\exp\left(-a\|\mathbf{x}\|^2\right)=0$ for every $a>0.$


\subsubsection{Limit Law for the Cavity Field}\label{LLCF}

Suppose $g_1,\ldots,g_N$ are i.i.d. standard Gaussian r.v.s independent of the randomness of $\left\{g_{ij}\right\}_{i<j\leq N}$. Define the {\it cavity field}
$l$ by
\begin{equation}\label{thm1:prep:eq1}
l=\frac{1}{\sqrt{N}}\sum_{j=1}^Ng_j\sigma_j.
\end{equation}
The name "cavity field" is due to the important role of $l$ played in the cavity method as we have already seen in Section $1.1.$
Note that the quenched average of $l$ is
\begin{equation}\label{thm1:prep:eq2}
r=\left<l\right>=\frac{1}{\sqrt{N}}\sum_{j=1}^Ng_j\left<\sigma_j\right>.
\end{equation}
The limit law of the centered distribution $l-r$ under the Gibbs
measure was firstly studied by Talagrand and it is approximately a centered Gaussian distribution with variance $1-q$.
The exact quantitative result is given by Theorem 1.7.11 in Talagrand \cite{Talag10}.

\begin{theorem}\label{thm0}$\left[4,\mbox{Theorem 1.7.11}\right]$
Let $\beta_0<1/2$ and $k\in\Bbb{N}.$ Suppose that $U$ is an infinitely differentiable function defined on $\Bbb{R}$ and
the derivatives of all orders of $U$ are of moderate growth.
Then for any $\beta\leq\beta_0$ and $h$, we have
\begin{align*}
E\left[\left<U\left(l-r\right)\right>-\int_{\Bbb{R}}U(x)\phi_{0,1-q}(x)dx\right]^{2k}\leq \frac{K}{N^k},
\end{align*}
where $K$ is a constant depending on $k,U,$ and $\beta_0$ only.

\end{theorem}

Note that here and the Theorem $\ref{thm1}$ and $\ref{thm2}$
below, we put strong condition on $U$ for convenience. In fact, as
we can see, from the proofs of Theorem $\ref{thm0}$, $\ref{thm1},$
and $\ref{thm2},$ this strong condition can be easily released.

By applying Theorem $\ref{thm0}$, Talagrand proved the TAP equations, see Theorem $2.4.20$ in \cite{Talag03} and
Theorem $1.7.7$ in \cite{Talag10}. However, his argument relies heavily on the special property of the exponential function
and it seems impossible to deduce the limit law for local fields from Theorem $\ref{thm0}.$ To overcome this difficulty, it would be very helpful
if we have good quantitative results for the limit law of $l,$ which is also one of the research problems
proposed by Talagrand (\cite{Talag10}, Research Problem $1.7.12$).
In our study we prove that the limit of $l$ is still concentrative and
under the Gibbs measure, $l$ is approximately Gaussian with mean $r$ and variance $1-q.$
Our main quantitative result is stated as follows.

\begin{theorem}\label{thm1}
Let $\beta_0<1/2$ and $k\in\Bbb{N}.$ Suppose that $U$ is an infinitely differentiable function defined on $\Bbb{R}$ and the derivatives of all orders
of $U$ are of moderate growth.
Then for any $\beta\leq\beta_0$ and $h,$ we have
\begin{equation}\label{thm1:eq1}
\begin{split}
&E\left[\left<U\left(l\right)\right>-\int_{\Bbb{R}}U(x)\phi_{r,1-q}(x)dx\right]^{2k}\leq \frac{K}{N^k},
\end{split}
\end{equation}
where $K$ is a constant depending on $k,U,$ and $\beta_0$ only.
\end{theorem}

In \cite{Talag03}, p. 87 and also \cite{Talag10}, Section 1.5, Talagrand gave intuitive arguments to support that the limit law of
the cavity field is approximately a Gaussian distribution with center $r$ and variance $1-q.$ Indeed, he proved that
when $k=1,$ the left-hand side of $(\ref{thm1:eq1})$
should be small without any error bound.
Later on Chatterjee \cite{Chatt10} based on Stein's method obtained the first quantitative result that
when $k=1$ and $U$ is a bounded measurable function $U$, the left-hand side of $(\ref{thm1:eq1})$ has an error bound
$c(\beta_0)\|U\|_\infty/\sqrt{N},$
where $c(\beta_0)$ is a constant depending on $\beta_0$ only. Hence, Theorem $\ref{thm1}$ justifies Talagrand's conjecture and
also improves Chatterjee's error bound if suitable smoothness on $U$ is assumed.
Informatively, as we will see, Theorem $\ref{thm1}$ is the bridge to study the limit law for local fields,
which is the main advantage that Theorem $\ref{thm0}$ does not have.


\subsubsection{Limit Law for the Local Fields}

Recall that the local field $l_i$ at site $i$ is defined by formula $(\ref{Intro:eq7}).$
For $1\leq i\leq N,$ suppose that $\nu_i$ is a random measure, whose density is the mixture of two Gaussian densities
$$
p_i\phi_{\gamma_i+\beta(1-q),1-q}+(1-p_i)\phi_{\gamma_i-\beta(1-q),1-q},
$$
where
\begin{align}\label{thm2:prep:eq1}
\gamma_i=&\frac{1}{\sqrt{N}}\sum_{j\leq N,j\neq i}g_{ij}\left<\sigma_j\right>-\beta(1-q)\left<\sigma_i\right>,\\\notag
\\\notag
p_i=&\frac{e^{\beta\gamma_i+h}}{e^{\beta \gamma_i+h}+e^{-\beta \gamma_i-h}}.\notag
\end{align}
Then we prove that the local field $l_i$ under the Gibbs measure is close to $\nu_i$ in the following sense:

\begin{theorem}\label{thm2}
Let $\beta_0<1/2$ and $k\in\Bbb{N}.$ Suppose that $U$ is an infinitely differentiable function defined on $\Bbb{R}$
and the derivatives of all orders of $U$ are of moderate growth.
Then for any $1\leq i\leq N$, $\beta\leq \beta_0$, and $h,$
\begin{align}\label{thm2:eq1}
E\left[\left<U(l_i)\right>-\int U(x)\nu_i(dx)\right]^{2k}\leq \frac{K}{N^k},
\end{align}
where $K$ is a constant depending on $\beta_0,k,$ and $U$ only.
\end{theorem}

Again, by applying Stein's method, Chatterjee \cite{Chatt10} proved the first quantitative result regarding the limit law for the local fields
that when $k=1$ and $U$ is a bounded measurable function $U,$ the left-hand side of $(\ref{thm2:eq1})$ has an error bound
$c(\beta_0)\|U\|_\infty /\sqrt{N}$,
where $c(\beta_0)$ is a constant depending on $\beta_0$ only. Thus, Theorem $\ref{thm2}$ improves this error bound
if we require some smoothness on $U.$ Recall formula $(\ref{Intro:eq5}).$ If
we put $U(x)=\tanh(\beta x+h)$ and modify the proof for Theorem $\ref{thm2}$ slightly,
then we obtain the same quantitative
result for the TAP equations in Talagrand $\left[4,\mbox{Theorem $1.7.7$}\right]:$

\begin{corollary}\label{cor2}
Let $\beta_0<1/2$ and $k\in\Bbb{N}.$ Then for each $1\leq i\leq N$, any $\beta\leq \beta_0,$ and $h,$ we have
$$
E\left[\left<\sigma_i\right>-\tanh\left(\frac{\beta}{\sqrt{N}}\sum_{j\leq N,j\neq i}g_{ij}\left<\sigma_j\right>+h-\beta^2(1-q)\left<\sigma_i\right>
\right)\right]^{2k}\leq \frac{K}{N^k},
$$
where $K$ is a constant depending on $\beta_0,$ and $k$ only.
\end{corollary}

\begin{proof}
Let us notice a useful formula from \cite{Chatt10}, equation $(9),$ which is stated as follows.
Suppose that $X$ is a random variable, whose density is the mixture of two Gaussian densities:
$p\phi_{\mu_1,\sigma^2}+(1-p)\phi_{\mu_2,\sigma^2}$ with $\mu_1>\mu_2$, $\sigma>0,$ and $0<p<1.$ Set
$$
a=\frac{\mu_1-\mu_2}{2\sigma^2},\,\,b=\frac{1}{2}\log\frac{p}{1-p}-\frac{\mu_1^2-\mu_2^2}{4\sigma^2}.
$$
Then a straightforward computation yields
\begin{align}\label{cor2:eq1}
E\left[\tanh (aX+b)\right]=\tanh(aE\left[X\right]+b-(2p-1)a^2\sigma^2).
\end{align}
Hence, if we apply Theorem $\ref{thm2}$ and use $(\ref{cor2:eq1})$ with $a=\beta$ and $b=h$, then we get the announced result.
Note that the constant $K$ still does not depend on $h,$ which can be verified by going through the proof for Theorem $\ref{thm2}$
and using the uniform boundedness of $U(x)=\tanh(\beta x+h).$
\end{proof}


\section{Proofs}\label{proofs}

This section is devoted to proving Theorems $\ref{thm1}$ and $\ref{thm2}.$ In the following proofs, the constants $K$, without mentioning
specifically, are always assumed to satisfy the requirements of the corresponding announced statements. Note that we use
$E_\zeta$ to stand for the expectation with respect to the randomness of $\zeta$ when $\zeta$ is a random variable.

\subsection{Proof of Theorem $\ref{thm1}$}

Using replicas, we set for $1\leq i\leq 2k,$
\begin{equation}\label{thm1:prep:eq4}
l^i=\frac{1}{\sqrt{N}}\sum_{j=1}^Ng_j\sigma_j^i.
\end{equation}
Suppose that $\xi,\xi^1,\ldots,\xi^{2k}$ are i.i.d. Gaussian r.v.s with mean zero and variance $1-q$
and they are independent of $\left\{g_j\right\}_{j\leq N}$ and $\left\{g_{ij}\right\}_{i<j\leq N}.$ Recalling definitions $(\ref{thm1:prep:eq1}),$
$(\ref{thm1:prep:eq2}),$ and $(\ref{thm1:prep:eq4}),$ we consider the Gaussian interpolations,
\begin{equation}\label{thm1:prep:eq5}
u(t)=\sqrt{t}(l-r)+\sqrt{1-t}\xi,\quad u_i(t)=\sqrt{t}(l^i-r)+\sqrt{1-t}\xi^i,\,\,1\leq i\leq 2k.
\end{equation}
Suppose that $U$ is a real-valued function defined on $\Bbb{R}$ and is of moderate growth. Define
\begin{equation}\label{thm1:prep:eq6}
V(x,y)=U(x+y)-E_\xi\left[U(\xi+y)\right]
\end{equation}
and
\begin{equation}\label{thm1:prep:eq7}
\psi(t)=E\left<\prod_{i\leq 2k}V(u_i(t),r)\right>.
\end{equation}
Notice that $\psi(0)=0$ and $\psi(1)$ is equal to the left-hand side of $(\ref{thm1:eq1}).$ Now, the main idea used to prove Theorem $\ref{thm1}$
is to control the derivatives of $\psi$ up to some optimal order by using Gaussian integration by parts.
In order to justify the application of this calculus result, Lemma $\ref{lem1}$ is needed. It says that mild composition and integration of functions of
moderate growth are still of moderate growth. On the other hand, to differentiate $\psi$, we need the differentiability of $V$, which is the
reason why we establish Lemma $\ref{lem2}$ below.


\begin{lemma}\label{lem1}
Suppose that $U,V:\Bbb{R}\rightarrow\Bbb{R}$ are of moderate growth.
Define $U_0(x)=U(x)+V(x)$, $U_1(x)=U(x)^k$, $U_2(x)=U(rx)$, $U_3(x)=U(x+r)$, $U_4(x,y)=U(x+y)$, and $U_5(x)=E_\xi \left[U(\xi+x)\right]$
for $k\in\Bbb{N}$, $r\in\Bbb{R}$, and $\xi$ a centered Gaussian r.v. with variance $\sigma^2.$
Then all of the functions defined above are of moderate growth.
\end{lemma}

\begin{proof}
It is easy to see that $U_0,$ $U_1$, and $U_2$ are of moderate growth. For $U_3$, since as $|x|\rightarrow\infty,$
$$
e^{-ax^2}U_3(x)=\exp\left({-\frac{a}{2}(x+r)^2}\right)U(x+r)\exp\left(-\frac{a}{2}(x^2-2rx-r^2)\right)\rightarrow 0,
$$
for all $a>0,$ it follows that $U_3$ is also of moderate growth. The function $U_4$ is of moderate growth since
\begin{align*}
&\limsup_{x^2+y^2\rightarrow \infty}|U_4(x,y)|\exp\left(-a(x^2+y^2)\right)\\
\leq &\limsup_{x^2+y^2\rightarrow \infty}1_{\left\{|x+y|\geq M\right\}}
|U(x+y)|\exp\left(-a(x^2+y^2)\right)\\
&+\limsup_{x^2+y^2\rightarrow \infty}1_{\left\{|x+y|<M\right\}}
|U(x+y)|\exp\left(-a(x^2+y^2)\right)\\
=&\limsup_{x^2+y^2\rightarrow \infty}1_{\left\{|x+y|\geq M\right\}}
|U(x+y)|\exp\left(-a(x^2+y^2)\right)\\
\leq &\limsup_{x^2+y^2\rightarrow \infty}1_{\left\{|x+y|\geq M\right\}}
|U(x+y)|\exp\left(-\frac{a}{2}(x+y)^2\right),
\end{align*}
for all $M>0$ and also the fact that $U$ is of moderate growth. Since $U_3$ is of moderate growth, $U_5$ is well-defined.
For $0<b<1,$ write
\begin{align*}
(x-y)^2-by^2=&\left(\sqrt{1-b}y-\frac{1}{\sqrt{1-b}}x\right)^2-\frac{b}{1-b}x^2.
\end{align*}
Thus,
\begin{align}\label{lem1:eq2}
\begin{split}
&\exp\left(-ax^2\right)U_5(x)\\
=&\frac{1}{\sqrt{2\pi \sigma^2}}\int_{-\infty}^\infty U(y)\exp\left(-ax^2\right)\exp\left(-\frac{(x-y)^2}{2\sigma^2}\right)dy\\
=&\frac{1}{\sqrt{2\pi \sigma^2}}\int_{-\infty}^\infty U(y)\exp\left(-\frac{by^2}{2\sigma^2}\right)\exp\left(-\frac{(x-y)^2-by^2}{2\sigma^2}-ax^2\right)dy\\
=&\frac{1}{\sqrt{2\pi \sigma^2}}\int_{-\infty}^\infty U(y)\exp\left(-\frac{by^2}{2\sigma^2}\right)
\exp\left(-\frac{1}{2\sigma^2}\left(\sqrt{1-b}y-\frac{1}{\sqrt{1-b}}x\right)^2\right)dy\\
&\times\exp\left(-\left(a-\frac{b}{2\sigma^2(1-b)}\right)x^2\right).
\end{split}
\end{align}
Since $U$ is of moderate growth, $U(y)\exp\left(-\frac{by^2}{2\sigma^2}\right)$ can be regarded as a bounded function in $y.$
On the other hand, since
$$
\int_{-\infty}^\infty\exp\left(-\frac{1}{2\sigma^2}\left(\sqrt{1-b}y-\frac{1}{\sqrt{1-b}}x\right)^2\right)dy
$$
is finite and independent of $x$, by taking $b$ to be small enough and letting $|x|$ tend to infinity, we conclude, from $(\ref{lem1:eq2}),$
that $U_5$ is of moderate growth. This completes the proof.

\end{proof}


\begin{lemma}\label{lem2}
Let $U:\Bbb{R}\rightarrow\Bbb{R}$ be continuously differentiable. Suppose that $U$ and its derivative $U'$ are of moderate growth.
Define $\psi(x)=E_\xi \left[U(\xi+x)\right]$ for $\xi$ a centered Gaussian r.v. with variance $\sigma^2.$ Then $\psi$ is differentiable and
$\psi'(x)=E_\xi \left[U'(\xi+x)\right].$
\end{lemma}

\begin{proof}
For any $x,x',y\in\Bbb{R},$ by the mean value theorem, we can find some $z(x,x',y)$ between $x$ and $x'$ so that
$U(x+y)-U(x'+y)=U'(z(x,x',y)+y)(x-x').$
Since $U'$ is of moderate growth, for any $M_1,M_2>0,$
\begin{align*}
K_1\equiv\sup_{|y|\geq M_1,|z|\leq M_2}|U'(z+y)|\exp\left(-\frac{y^2}{4\sigma^2}\right)<\infty.
\end{align*}
By the continuity of $U',$ $K_2\equiv\sup_{|y|\leq M_1,|z|\leq M_2}|U'(z+y)|\exp\left(-y^2/2\sigma^2\right)<\infty.$
Therefore, for $|x|,|x'|\leq M_2,$
\begin{align*}
&\left|\frac{U(x+y)-U(x'+y)}{x-x'}\right|\exp\left(-\frac{y^2}{2\sigma^2}\right)\\
\leq &K_1\exp\left(-\frac{y^2}{4\sigma^2}\right)
1_{\left\{|y|\geq M_1\right\}}+K_21_{\left\{|y|<M_2\right\}}
\end{align*}
and by the dominated convergence theorem, $\psi'(x)=E_\xi U'(\xi+x).$

\end{proof}


In view of Theorems $\ref{thm0}$, $\ref{thm1}$, and $\ref{thm2}$, the constants $K$ are independent of $N$ and $\beta.$
The main reason is because, when conditioning on the randomness of $\left\{g_{ij}\right\}_{i<j\leq N}$,
the cavity field and its quenched average are centered Gaussian distributions, whose variances are bounded above by some constants, which are
independent of $\left\{g_{ij}\right\}_{i<j\leq N}$, $N$, and $\beta.$ Therefore, we still have good control on
the moment estimates of the cavity field and its quenched average.
This observation will be used repeatedly and for convenience,
we formulate it as Lemma $\ref{lem3}.$

\begin{lemma}\label{lem3}
Let $I,J\subset\left[0,\infty\right)$ be two bounded intervals.
Suppose that $K_1,K_2,K_3$ are positive constants and that the following assumptions hold.
\begin{itemize}
\item[i)] Let $z,g_1^N,\ldots,g_N^N$ be i.i.d. standard Gaussian r.v.s for $N\in\Bbb{N}.$

\item[ii)] For $N\in\Bbb{N},$ suppose that
$$\left\{X_{j,\beta}^N:1\leq j\leq N,\beta\in J\right\}$$ is a family of random variables such that
$
|X_{j,\beta}^N|\leq K_1
$
for $1\leq j\leq N,$ and $\beta \in J.$

\item[iii)] Let $f_{1},f_2:\Bbb{N}\times I\times J\rightarrow\Bbb{R}$ be measurable functions such that
$$
|f_{1}(N,t,\beta)|\leq K_2/\sqrt{N},\quad|f_2(N,t,\beta)|\leq K_3,
$$
for $(N,t,\beta)\in \Bbb{N}\times I\times J.$
\item[iv)] Let $U:\Bbb{R}\rightarrow\Bbb{R}$ be a continuous function. Suppose that there are some
$A>0$ and $0<a<\min\left\{(8K_1^2K_2^2)^{-1},(4K_3^2)^{-1}\right\}$
such that $|U(x)|\leq A\exp(ax^2)$ for all $x\in\Bbb{R}.$

\end{itemize}
Then there is a constant $K>0$ such that
\begin{align}\label{lem3:eq1}
\sup_{N\in\Bbb{N},\beta\in J}E_0\left[\sup_{t\in I}\left|U\left(
f_{1}(N,t,\beta)\sum_{j\leq N}g_j^NX_{j,\beta}^N+f_{2}(N,t,\beta)z
\right)\right|\right]\leq K,
\end{align}
where $E_0$ means the expectation with respect to $\left\{g_j^N:j\leq N,\,N\in\Bbb{N}\right\}$ and $z.$
\end{lemma}

\begin{proof}
From the given conditions, we obtain
\begin{align*}
&\left(f_{1}(N,t,\beta)\sum_{j\leq N}g_j^NX_{j,\beta}^N+f_{2}(N,t,\beta)z\right)^2\\
\leq &2f_1(N,t,\beta)^2\left(\sum_{j\leq N}g_{j}^NX_{j,\beta}^N\right)^2
+2f_2(N,t,\beta)^2z^2\\
\leq &4f_1(N,t,\beta)^2\sum_{j\leq N}\left(g_{j}^N\right)^2\left(X_{j,\beta}^N\right)^2+2K_3^2z^2\\
\leq &4K_1^2K_2^2\left(\frac{1}{N}\sum_{j\leq N}\left(g_j^N\right)^2\right)+2K_3^2z^2
\end{align*}
and so
\begin{align*}
&\sup_{t\in I}\left|U\left(f_{1}(N,t,\beta)\sum_{j\leq N}g_j^NX_{j,\beta}^N+f_{2}(N,t,\beta)z
\right)\right|\\
\leq &A\exp\left(4aK_1^2K_2^2\left(\frac{1}{N}\sum_{j\leq N}\left(g_j^N\right)^2\right)+2aK_3^2z^2\right).
\end{align*}
Now integrating this inequality, we find the left-hand side of $(\ref{lem3:eq1})$ is then bounded above by
\begin{align*}
E_0\left[\exp\left(4aK_1^2K_2^2(g_1^N)^2\right)\right]E_0\left[\exp\left(2aK_3^2z^2\right)\right],
\end{align*}
which is a finite constant independent of $N$ and $\beta$ since $a$ satisfies $iv).$

\end{proof}



In the sequel, we use $E_0$ to denote the expectation with respect to the randomness of $\left\{g_j\right\}_{j\leq N}$ and
$\left\{\xi^i\right\}_{i\leq 2k}.$ Recall formulas $(\ref{thm1:prep:eq2})$ and $(\ref{thm1:prep:eq5})$ for $r$ and $u_i,$
respectively. The following lemma,
as an application of Gaussian integration by parts, is our main equation to control the derivatives of all orders of $\psi.$

\begin{lemma}\label{lem4}
Let $k\in\Bbb{N}.$ Suppose that $V_1,V_2,\ldots,V_{2k}:\Bbb{R}^2\rightarrow\Bbb{R}$ are twice continuously differentiable
functions and their first and second order partial derivatives
are of moderate growth. Define
$$
\varphi(t)=E_0\left[\prod_{i\leq 2k}V_i(u_i(t),r)\right],\quad 0\leq t\leq 1.
$$
Then $\varphi$ is differentiable on $\left(0,1\right)$ and
\begin{align}\label{lem4:eq1}
\begin{split}
\varphi'(t)=&\frac{1}{2}\sum_{i\leq 2k}T_{i,i}E_0\left[\frac{\partial^2V_i}{\partial x^2}(u_i(t),r)\prod_{j\leq 2k,j\neq i}V_j(u_j(t),r)\right]\\
&+\frac{1}{2}\sum_{i,j\leq 2k,i\neq j}T_{i,j}E_0\left[\frac{\partial V_i}{\partial x}(u_i(t),r)\frac{\partial V_j}{\partial x}(u_j(t),r)
\prod_{h\leq 2k,h\neq i,j}V_h(u_h(t),r)\right]\\
&+\frac{1}{2\sqrt{t}}\sum_{i\leq 2k}T_iE_0\left[\frac{\partial^2V_i}{\partial x\partial y}(u_i(t),r)\prod_{j\leq 2k,j\neq i}V_j(u_j(t),r)\right]\\
&+\frac{1}{2\sqrt{t}}\sum_{i,j\leq 2k,i\neq j}
T_iE_0\left[\frac{\partial V_i}{\partial x}(u_i(t),r)\frac{\partial V_j}{\partial y}(u_j(t),r)\prod_{h\leq 2k,h\neq i,j}V_h(u_h(t),r)\right].
\end{split}
\end{align}
Here, $\frac{\partial^{a+b}V_i}{\partial^a x\partial^b y}$ means that we differentiate $V_i$ with respect to the first variable $a$ times and
with respect to the second variable $b$ times.
\end{lemma}

\begin{proof}
To prove the differentiability of $\varphi$, it suffices to prove, with the help of the mean value theorem and
the dominated convergence theorem, that for $0<\delta<1/2$ and $1\leq i\leq 2k,$
\begin{align}\label{lem4:eq0}
E_0\left[\sup_{\delta\leq t\leq 1-\delta}|u_i'(t)|\sup_{0\leq t\leq 1}\left|\frac{\partial V_i}{\partial x}(u_i(t),r)\right|
\prod_{j\leq 2k,j\neq i}\sup_{0\leq t\leq 1}\left|V_j(u_j(t),r)\right|\right]\leq K,
\end{align}
for some constant $K.$
 Note that $$
u_i'(t)=\frac{1}{2\sqrt{Nt}}\sum_{j\leq N}g_j\dot{\sigma}_j-\frac{1}{2\sqrt{1-t}}\xi^i.
$$
Since $\frac{\partial V_i}{\partial x}$
and $V_j$ are of moderate growth, for any $a>0$, there exists some $A>0$ such that
\begin{align}\label{lem4:eq4}
\left|\frac{V_i}{\partial x}(x,y)\right|,\left|V_j(x,y)\right|\leq A\exp(a(x^2+y^2)),\quad (x,y)\in\Bbb{R}^2.
\end{align}
Set $\left(z,g_j^N\right)=\left((1-q)^{-1}\xi^i,g_j\right)$ and also let
$\left(X_{j,\beta}^N,f_1(N,t,\beta),f_2(N,t,\beta),U(x)\right)$ be any one of the following vectors
\begin{align*}
&\left(\dot{\sigma}_j,\frac{1}{2\sqrt{Nt}},-\frac{1-q}{2\sqrt{1-t}},x^{2k+1}\right),\\
&\left(\dot{\sigma}_j,\frac{\sqrt{t}}{\sqrt{N}},(1-q)\sqrt{1-t},\exp((4k+2)ax^2)\right),\\
&\left(\left<\sigma_j\right>,\frac{1}{\sqrt{N}},0,\exp((4k+2)ax^2)\right).
\end{align*}
Then by choosing $a$ small enough and applying Lemma $\ref{lem3}$, there exists a constant $K$ independent of $\beta$ and $N$ such that
\begin{align*}
&E_0\left[\sup_{\delta\leq t\leq 1-\delta}|u_i'(t)|^{2k+1}\right]\leq K,\\
&E_0\left[\sup_{0\leq t\leq 1}\exp\left((4k+2)au_i(t)^2\right)\right]\leq K,\\
&E_0\left[\sup_{0\leq t\leq 1}\exp\left((4k+2)ar^2\right)\right]\leq K.
\end{align*}
Therefore, from $(\ref{lem4:eq4})$ and Cauchy-Schwarz inequality,
$$
E_0\left[\sup_{0\leq t\leq 1}\left|\frac{\partial V_i}{\partial x}(u_i(t),r)\right|\right]^{2k+1},\,\,
E_0\left[\sup_{0\leq t\leq 1}\left|V_j(u_j(t),r)\right|\right]^{2k+1}\leq K
$$
and by H\"{o}lder's inequality, $(\ref{lem4:eq0})$ holds.

To prove $(\ref{lem4:eq1}),$ we use Gaussian integration by parts,
\begin{align}\label{lem4:eq2}
\begin{split}
&\varphi'(t)\\
=&\sum_{i\leq 2k}E_0\left[u_i'(t)\frac{\partial V_i}{\partial x}(u_i(t),r)\prod_{j\leq 2k,j\neq i}V_j(u_j(t),r)\right]\\
=&\sum_{i\leq 2k}E_0\left[u_i'(t)u_i(t)\right]E_0\left[\frac{\partial^2 V_i}{\partial x^2}(u_i(t),r)\prod_{j\leq 2k,j\neq i}V_j(u_j(t),r)\right]\\
&+\sum_{i,j\leq 2k,i\neq j}E_0\left[u_i'(t)u_j(t)\right]E_0\left[\frac{\partial V_i}{\partial x}(u_i(t),r)
\frac{\partial V_j}{\partial x}(u_j(t),r)\prod_{h\leq 2k,h\neq i,j}V_h(u_h(t),r)\right]\\
&+\sum_{i\leq 2k}E_0\left[u_i'(t)r\right]E_0\left[\frac{\partial^2 V_i}{\partial x\partial y}(u_i(t),r)
\prod_{j\leq 2k,j\neq i}V_j(u_j(t),r)\right]\\
&+\sum_{i,j\leq 2k,i\neq j}E_0\left[u_i'(t)r\right]E_0\left[\frac{\partial V_i}{\partial x}(u_i(t),r)
\frac{\partial V_j}{\partial y}(u_j(t),r)\prod_{h\leq 2k,h\neq i,j}V_h(u_h(t),r)\right].
\end{split}
\end{align}
Recalling definition $(\ref{Intro:eq4})$, a straightforward computation yields
\begin{align}\label{lem4:eq3}
\begin{split}
E_0\left[u'_i(t)u_i(t)\right]=&T_{i,i}/2,\\
E_0\left[u_i'(t)u_j(t)\right]=&T_{i,j}/{2},\quad\mbox{for $i\neq j$},\\
E_0\left[u_i'(t)r\right]=&{T_i}/{2\sqrt{t}}.
\end{split}
\end{align}
Combining $(\ref{lem4:eq2})$ and $(\ref{lem4:eq3})$ gives $(\ref{lem4:eq1}).$

\end{proof}

\begin{remark}\label{rem1}\rm
Formula $(\ref{lem4:eq1})$ implies that our computation on the derivative of $\varphi$ can be completely determined by
$T_{i,i}$, $T_{i,j}$, and $T_i$ in the following manner.
Each $T_{i,i}$ is associated with $\frac{\partial^2 V_i}{\partial x^2}(u_i(t),r)$; each $T_{i,j}$ is associated with
$\frac{\partial V_i}{\partial x}(u_i(t),r)\frac{\partial V_j}{\partial x}(u_j(t),r).$ As for each $T_i,$ it is associated with either
$\frac{\partial^2V_i}{\partial x\partial y}(u_i(t),r)$ or $\frac{\partial V_i}{\partial x}(u_i(t),r)\frac{\partial V_j}{\partial y}(u_j(t),r)$ for some $j\neq i.$
This observation will be very useful when we control the derivatives of all orders of $\psi.$
\end{remark}


Lemma $\ref{lem4}$ will be used iteratively up to some optimal order. Since on each iteration, equation $(\ref{lem4:eq1})$
brings us many terms, we will finally obtain a very huge number of summations. Therefore,
in order to make our argument clearer, we formulate the following result.

\begin{lemma}\label{lem:combin}
Fix an integer $m>0$
and let $\psi$ and $\psi_{\mathbf{s}_n}$
be real-valued smooth functions defined on $\left[0,1\right]$ for every $\mathbf{s}_n=(s_n(1),\ldots,s_n(n))\in \left\{0,1\right\}^n$ with $1\leq n\leq m+1.$
Suppose that $\psi(0)=0$ and $\psi_{\mathbf{s}_n}(0)=0$ for every $\mathbf{s}_n\in\left\{0,1\right\}^n$ with $1\leq n<m.$ If
\begin{align}\label{lem:combin:eq1}
\psi'(t)=\frac{1}{2}\psi_{(0)}(t)+\frac{1}{2\sqrt{t}}\psi_{(1)}(t)
\end{align}
and
\begin{align}\label{lem:combin:eq2}
\psi_{\mathbf{s}_n}'(t)=\frac{1}{2}\psi_{(\mathbf{s}_n,0)}(t)+\frac{1}{2\sqrt{t}}\psi_{(\mathbf{s}_n,1)}(t),
\end{align}
for every $\mathbf{s}_n\in\left\{0,1\right\}^n$ with $1\leq n\leq m,$ then
\begin{align}\label{lem:combin:eq3}
\begin{split}
\psi(t)=&
\frac{1}{2^m}\sum_{\mathbf{s}_{m}}\int_0^t\int_0^{t_1}\ldots\int_0^{t_{m-1}}\frac{1}{\prod_{n=1}^{m}t_n^{s_n/2}}dt_{m}
\ldots dt_2dt_1\psi_{\mathbf{s}_{m}}(0)\\
&+\frac{1}{2^{m+1}}\sum_{\mathbf{s}_{m+1}}\int_0^t\int_0^{t_1}\ldots\int_0^{t_{m}}\frac{1}{\prod_{n=1}^{m+1}t_n^{s_n/2}}\psi_{\mathbf{s}_{m+1}}(t_{m+1})dt_{m+1}
\ldots dt_2dt_1.
\end{split}
\end{align}
\end{lemma}

\begin{proof}
It suffices to prove that
\begin{align}\label{lem:combin:eq4}
\psi(t)=&\frac{1}{2^m}\sum_{\mathbf{s}_{m}}\int_0^t\int_0^{t_1}\ldots\int_0^{t_{m-1}}\frac{1}{\prod_{n=1}^{m}t_n^{s_n/2}}\psi_{\mathbf{s}_{m}}(t_{m})dt_{m}
\ldots dt_2dt_1.
\end{align}
Indeed, if $(\ref{lem:combin:eq4})$ holds, then $(\ref{lem:combin:eq3})$ can be deduced by applying $$
\psi_{\mathbf{s}_m}(t_m)=\int_0^{t_m}
\psi_{\mathbf{s}_m}'(t_{m+1})dt_{m+1}+\psi_{\mathbf{s}_m}(0)$$ and $(\ref{lem:combin:eq2})$ to $(\ref{lem:combin:eq4}).$

Let us prove $(\ref{lem:combin:eq4})$ by induction on $m$. If $m=1,$ from $(\ref{lem:combin:eq1}),$
$(\ref{lem:combin:eq4})$ holds clearly by
$$
\psi(t)=\int_0^t\psi'(t_1)dt_1+\psi(0)=\frac{1}{2}\int_0^t\left(\psi_{(0)}(t_1)+\frac{1}{\sqrt{t_1}}\psi_{(1)}(t_1)\right)dt_1.
$$
Suppose that the announced result is true for $m-1\geq 1$. Let $\psi$ and $\psi_{\mathbf{s}_n}$ be real-valued smooth functions
for every $\mathbf{s}_n\in \left\{0,1\right\}^{n}$ with $n\leq m+1$ satisfying the assumptions of this lemma.
Notice that from $(\ref{lem:combin:eq2}),$ we obtain
\begin{align*}
\psi_{\mathbf{s}_{m-1}}(t_{m-1})=&\int_0^{t_{m-1}}\psi_{\mathbf{s}_{m-1}}'(t_{m})dt_{m}+\psi_{\mathbf{s}_{m-1}}(0)\\
=&\frac{1}{2}\int_0^{t_{m-1}}\left(\psi_{(\mathbf{s}_{m-1},0)}(t_{m})+\frac{1}{\sqrt{t_{m}}}\psi_{(\mathbf{s}_{m-1},1)}(t_{m})\right)dt_{m}
\end{align*}
and also by induction hypothesis,
$$
\psi(t)=\frac{1}{2^{m-1}}\sum_{\mathbf{s}_{m-1}}\int_0^t\int_0^{t_1}\ldots\int_0^{t_{m-2}}\frac{1}{\prod_{n=1}^{m-1}t_n^{s_n/2}}\psi_{\mathbf{s}_{m-1}}(t_{m-1
})dt_{m-1}
\ldots dt_2dt_1.
$$
Now $(\ref{lem:combin:eq4})$ follows by combining last two equations together
\begin{align*}
\psi(t)=&\frac{1}{2^{m-1}}\sum_{\mathbf{s}_{m-1}}\int_0^t\int_0^{t_1}\ldots\int_0^{t_{m-2}}\int_0^{t_{m-1}}\frac{1}{\prod_{n=1}^{m-1}t_n^{s_n/2}}\\
&\qquad\qquad\times\left(\frac{1}{2}\psi_{(\mathbf{s}_{m-1},0)}(t_{m})+\frac{1}{2\sqrt{t_{m}}}\psi_{(\mathbf{s}_{m-1},1)}(t_{m})\right)
dt_{m}dt_{m-1}\ldots dt_2dt_1\\
=&\frac{1}{2^{m}}\sum_{\mathbf{s}_{m}}\int_0^t\int_0^{t_1}\ldots\int_0^{t_{m-1}}\frac{1}{\prod_{n=1}^{m}t_n^{s_n/2}}\psi_{\mathbf{s}_{m}}(t_{m})dt_{m-1}
\ldots dt_2dt_1.
\end{align*}

\end{proof}


\begin{Proof of theorem} $\mathbf{\ref{thm1}}:$
Basically our proof is based on the same idea as the proof of Theorem $\ref{thm0}.$
Recall formulas $(\ref{thm1:prep:eq6})$ and $(\ref{thm1:prep:eq7})$ for $V$ and $\psi$, respectively.
From Lemmas $\ref{lem1}$ and $\ref{lem2},$ $V$ is an infinitely differentiable function and the partial derivatives of all orders of $V$ are of moderate growth.
We also note that $\psi$ is infinitely differentiable
by applying the same argument as Lemma $\ref{lem4}$. Recall the definition of $E_0$ and use Fubini's theorem, $\psi(t)$ is equal to
$$
E\left<E_0\left[\prod_{i\leq 2k}V(u_i(t),r)\right]\right>.
$$
Now applying Lemma $\ref{lem4}$ to
$$
E_0\left[\prod_{i\leq 2k}V(u_i(t),r)\right]
$$
and then taking expectation $E\left<\cdot\right>$ on $(\ref{lem4:eq1}),$
we obtain
\begin{align*}
\psi'(t)=&\frac{1}{2}\psi_{(0)}(t)+\frac{1}{2\sqrt{t}}\psi_{(1)}(t),
\end{align*}
where
\begin{align*}
\psi_{(0)}(t)=&\sum_{i\leq 2k}E\left<T_{i,i}E_0\left[\frac{\partial^2V}{\partial x^2}(u_i(t),r)\prod_{j\leq 2k,j\neq i}V(u_j(t),r)\right]\right>\\
&+\sum_{i,j\leq 2k,i\neq j}E\left<T_{i,j}E_0\left[\frac{\partial V}{\partial x}(u_i(t),r)\frac{\partial V}{\partial x}(u_j(t),r)
\prod_{h\leq 2k,h\neq i,j}V(u_h(t),r)\right]\right>
\end{align*}
and
\begin{align*}
\psi_{(1)}(t)=&\sum_{i\leq 2k}E\left<T_iE_0\left[\frac{\partial^2V}{\partial x\partial y}(u_i(t),r)\prod_{j\leq 2k,j\neq i}V(u_j(t),r)\right]\right>\\
&+\sum_{i,j\leq 2k,i\neq j}
E\left<T_iE_0\left[\frac{\partial V}{\partial x}(u_i(t),r)\frac{\partial V}{\partial y}(u_j(t),r)\prod_{h\leq 2k,h\neq i,j}V(u_h(t),r)\right]\right>.
\end{align*}
Next, to compute the derivative of $\psi_{(0)}$, let us apply Lemma $\ref{lem4}$ again to
$$
E_0\left[\frac{\partial^2V}{\partial x^2}(u_i(t),r)\prod_{j\leq 2k,j\neq i}V(u_j(t),r)\right],\quad i\leq 2k,
$$
and
$$
E_0\left[\frac{\partial V}{\partial x}(u_i(t),r)\frac{\partial V}{\partial x}(u_j(t),r)
\prod_{h\leq 2k,h\neq i,j}V(u_h(t),r)\right],\quad i,j\leq 2k,\,\,i\neq j,
$$
and then take expectation $E\left<\cdot\right>$. Then
we obtain two smooth functions $\psi_{(0,0)}$ and $\psi_{(0,1)}$ defined on $\left[0,1\right]$ such that
$$
\psi_{(0)}'(t)=\frac{1}{2}\psi_{(0,0)}(t)+\frac{1}{2\sqrt{t}}\psi_{(0,1)}(t).
$$
Similarly, the derivative of $\psi_{(1)}'$ can also be computed in the same way, which leads to two smooth functions $\psi_{(1,0)}$ and $\psi_{(1,1)}$
defined on $\left[0,1\right]$ such that
$$
\psi_{(1)}'(t)=\frac{1}{2}\psi_{(1,0)}(t)+\frac{1}{2\sqrt{t}}\psi_{(1,1)}(t).
$$
Continuing this process, we get $\left\{\psi_{\mathbf{s}_n}:s_n\in\left\{0,1\right\}^n,n\leq 2k+1\right\}$
for which $(\ref{lem:combin:eq1})$ and $(\ref{lem:combin:eq2})$ hold.

We claim that $\psi_{\mathbf{s}_n}(0)=0$ for every $\mathbf{s}_n=(s_n(1),\ldots,s_n(n))\in\left\{0,1\right\}^n$ with $n<2k.$
To see this, let us observe that from Remark $\ref{rem1},$ a typical term in those very complicated summations in the expression
of $\psi_{\mathbf{s}_n}$ is of the form
\begin{align}\label{lem1:eq1}
E\left<\prod_{i\leq 2k}T_{i,i}^{k_1(i)}\prod_{i,j\leq 2k,i\neq j}T_{i,j}^{k_2(i,j)}\prod_{i\leq 2k}T_i^{k_3(i)}E_0\left[
\prod_{i\leq 2k}\frac{\partial^{k_4(i)+k_5(i)}V}{\partial x^{k_4(i)}\partial y^{k_5(i)}}(u_i(t),r)\right]\right>,
\end{align}
with
\begin{align*}
&\sum_{i\leq 2k}k_1(i)+\sum_{i,j\leq 2k,i\neq j}k_2(i,j)+\sum_{i\leq 2k}k_3(i)=n,\\
&2k_1(i)+\sum_{j\leq 2k,j\neq i}(k_2(i,j)+k_2(j,i))+k_3(i)=k_4(i),\,\,i\leq 2k,\\
&\sum_{i\leq 2k}k_3(i)=\sum_{i\leq n}s_n(i)=\sum_{i\leq 2k}k_5(i),
\end{align*}
where $k_1(i),k_2(i,j),k_3(i),k_4(i)$, and $k_5(i)$ are nonnegative integers for $i,j\leq 2k$ with $i\neq j.$

First of all, notice that if
$$
\left<\prod_{i\leq 2k}T_{i,i}^{k_1(i)}\prod_{i,j\leq 2k,i\neq j}T_{i,j}^{k_2(i,j)}\prod_{i\leq 2k}T_i^{k_3(i)}\right>\neq 0,
$$
then for $i\leq 2k,$ either
$k_4(i)\geq 2$ or it is equal to zero. That is, if $i$ occurs in one of the subscripts of
$T_{i,i},$ $T_{i,j}$ or $T_i$, it must occur more than once and we suppose that this is the case. Secondly,
if
$$
E_0\left[
\prod_{i\leq 2k}\frac{\partial^{k_4(i)+k_5(i)}V}{\partial x^{k_4(i)}\partial y^{k_5(i)}}(u_i(0),r)\right]\neq 0,
$$
then $k_4(i)\geq 1$ for every $i\leq 2k$ since $E_{\xi_i}\left[\frac{\partial^{h}V}{\partial y^{h}}(u_i(0),r)\right]=0$ for all $h\geq 0$ and $i\leq 2k.$
Therefore, we conclude that
$k_4(i)\geq 2$
for every $i\leq 2k$ and it implies
\begin{align*}
2n=&2\sum_{i\leq 2k}k_1(i)+2\sum_{i,j\leq 2k,i\neq j}k_2(i,j)+2\sum_{i\leq 2k}k_3(i)\\
=&\sum_{i\leq 2k}\left[2k_1(i)+\sum_{j\leq 2k,j\neq i}(k_2(i,j)+k_2(j,i))+k_3(i)\right]+\sum_{i\leq 2k}k_3(i)\\
=&\sum_{i\leq 2k}k_4(i)+\sum_{i\leq 2k}k_3(i)
\\
\geq & 2k\cdot 2+\sum_{i\leq 2k}k_3(i)\\
\geq &2k\cdot 2\\
=&4k.
\end{align*}
Hence, if $(\ref{lem1:eq1})$ is not equal to zero, then $n\geq 2k$. So
$\psi_{\mathbf{s}_n}(0)=0$ for every ${\mathbf{s}_n}\in\left\{0,1\right\}^{n}$ with $n<2k,$ which completes the proof of our claim.
In addition, we can conclude more from above that if $n=2k$ and $(\ref{lem1:eq1})$ does not vanish, since $\sum_{i\leq 2k}k_3(i)=\sum_{i\leq n}s_n(i)$
and $k_4(i)\geq 2$ for every $i\leq 2k,$
it implies that $s_n(i)=0$ and $k_4(i)=2$ for every $i\leq 2k.$ Consequently, it means that $\psi_{\mathbf{s}_{2k}}(0)=0$ for every $\mathbf{s}_{2k}
\in\left\{0,1\right\}^{2k}$ unless $\mathbf{s}_{2k}=\mathbf{0}_{2k}\equiv(0,\ldots,0)$ and in this case, $$
\psi_{\mathbf{0}_{2k}}(0)=\sum_{i_1,i_1',\ldots,i_{2k},i_{2k}'} E\left<T_{i_1,i_1'}\cdots T_{i_{2k},i_{2k}'}E_0
\left[\prod_{i\leq 2k}U''(\xi^i+r)\right]\right>,
$$
where $(i_1,i_1',\ldots,i_{2k},i_{2k}')\in\left\{1,\ldots,2k\right\}^{4k}$ satisfies that each number $i\leq 2k$ occurs exactly twice in this vector.

Now, concluding from $(\ref{lem:combin:eq3})$ in Lemma $\ref{lem:combin}$ and using
$$
\int_0^t\int_0^{t_1}\ldots\int_0^{t_{2k-1}}dt_{2k}\ldots dt_2dt_1=\frac{t^{2k}}{(2k)!},
$$
we obtain
\begin{align*}
\psi(t)=&
\frac{t^{2k}}{2^{2k}(2k)!}\sum_{i_1,i_1',\ldots,i_{2k},i_{2k}'} E\left<T_{i_1,i_1'}\cdots T_{i_{2k},i_{2k}'}E_0
\left[\prod_{i\leq 2k}U''(\xi^i+r)\right]\right>\\
&+\frac{1}{2^{2k+1}}\sum_{\mathbf{s}_{2k+1}}\int_0^t\int_0^{t_1}\ldots\int_0^{t_{2k}}\frac{1}{\prod_{n=1}^{2k+1}t_n^{s_n/2}}\psi_{\mathbf{s}_{
2k+1}}(t_{2k+1})dt_{2k+1}
\ldots dt_2dt_1.\notag
\end{align*}
Finally, since each term in those very complicated summations in the expression of $\psi_{\mathbf{s}_{2k+1}}$ is given by formula $(\ref{lem1:eq1})$ with $n=2k+1,$
by applying Lemma $\ref{lem3}$, the known result $(\ref{Intro:eq3})$, and H\"{o}lder's inequality, we obtain some $K>0$ depending on $\beta_0$, $k$, and $U$
only such that
$\sup_{0\leq t\leq 1}|\psi_{s_{2k+1}}(t)|\leq \frac{K}{N^{k+1/2}}$ for every $\beta\leq \beta_0$ and $h.$
Similarly we also have
$$
\left|E\left<T_{i_1,i_1'}\cdots T_{i_{2k},i_{2k}'}E_0\left[\prod_{i\leq 2k}U''(\xi^i+r)\right]\right>\right|\leq \frac{K}{N^{k}}.
$$
Therefore, $\psi(1)\leq \frac{K}{N^k}$
and we are done.
\end{Proof of theorem}


\subsection{Proof of Theorem \ref{thm2}}




The following proposition is the key to proving Theorem $\ref{thm2}.$

\begin{proposition}\label{prop1}
Let $\beta_0<1/2$ and $k\in\Bbb{N}.$ Suppose that $U$ is an infinitely differentiable function defined on $\Bbb{R}$
and the derivatives of all orders of $U$ are of moderate growth.
Recall $l$ and $r$ as defined by $(\ref{thm1:prep:eq1})$ and $(\ref{thm1:prep:eq2}).$
Then for any $\beta\leq \beta_0$ and $h,$
\begin{equation*}
E\left[\frac{\left<U(l)\cosh(\beta l+h)\right>}{\left<\cosh(\beta l+h)\right>}-\frac{E_\xi\left[U(\xi+r)
\cosh(\beta(\xi+r)+h)\right]}{\exp\left(\frac{\beta^2}{2}(1-q)\right)\cosh(\beta r+h)}\right]^{2k}\leq \frac{K}{N^k},
\end{equation*}
where $\xi$ is a centered Gaussian distribution with variance $1-q$ and $K$ depends on $\beta_0,k$ and $U$ only.
\end{proposition}

\begin{proof}
Define for $\varepsilon=\pm1,$
$$
A(\varepsilon)=\left<U(l)\exp\left(\varepsilon \beta l\right)\right>-E_\xi\left[U(\xi+r)\exp\left(\varepsilon \beta(\xi+r)\right)\right]
$$
and also
$$
B(\varepsilon)=\left<\exp\left(\varepsilon \beta l\right)\right>-E_\xi\left[\exp\left(\varepsilon \beta(\xi+r)\right)\right].
$$
Then $EA(\varepsilon)^{4k}\leq K/N^{2k}$ and $EB(\varepsilon)^{8k}\leq K/N^{4k}$ by applying the known result (\ref{Intro:eq3}), Lemma $\ref{lem3}$,
and H\"{o}lder's inequality
to Theorem $\ref{thm1}.$ Now by applying Jensen's inequality, H\"{o}lder's inequality, and Lemma $\ref{lem3},$
it implies that
\begin{align}\label{prop1:eq1}
E\left[\frac{1}{\left<\exp\left(\varepsilon\beta l\right)\right>^{4k}}\right]\leq E\left[\exp\left(-4k\varepsilon\beta r\right)\right]\leq K,
\end{align}
\begin{align}\label{prop1:eq2}
E\left[\frac{1}{\exp\left(\varepsilon\beta r\right)^{8k}}\right]=E\left[\exp\left(-8k\varepsilon\beta r\right)\right]\leq K,
\end{align}
and
\begin{align}\label{prop1:eq3}
E\left[U(\xi+r)^{4k}\frac{\exp(4k\varepsilon\beta(\xi+r))}{\left<\exp(\varepsilon\beta l)\right>^{4k}}\right]
\leq E\left[U(\xi+r)^{4k}\exp(4k\varepsilon\beta\xi)\right]\leq K.
\end{align}
For convenience, we set
$$
\begin{array}{rcl}
A&=&\left<U(l)\cosh(\beta l+h)\right>,\\
\\
B&=&\left<\cosh(\beta l+h)\right>,\\
\\
A'&=&E_\xi\left[U(\xi+r)\cosh(\beta(\xi+r)+h)\right],\\
\\
B'&=&\exp\left(\frac{\beta^2}{2}(1-q)\right)\cosh(\beta r+h).
\end{array}
$$
Consequently, by using H\"{o}lder's inequality, $(\ref{prop1:eq1}),(\ref{prop1:eq2}),$ and $(\ref{prop1:eq3}),$ the following three inequalities hold
\begin{align}\label{prop1:eq4}
\begin{split}
E\left[\frac{A-A'}{B}\right]^{2k}
=&E\left[\frac{A(1)e^{h}+A(-1)e^{-h}}{\left<\exp\left(\beta l\right)\right>e^h+\left<\exp\left(-\beta l\right)\right>e^{-h}}\right]^{2k}\\
\leq &2^{2k}\left(E\left[\frac{A(1)}{\left<\exp\left(\beta l\right)\right>}\right]^{2k}+
E\left[\frac{A(-1)}{\left<\exp\left(-\beta l\right)\right>}\right]^{2k}\right)\\
\leq &2^{2k}\left(\left(EA(1)^{4k}\right)^{1/2}\left(E\left[\frac{1}{\left<\exp\left(\beta l\right)\right>^{4k}}\right]\right)^{1/2}\right.\\
&+\left.
\left(EA(-1)^{4k}\right)^{1/2}\left(E\left[\frac{1}{\left<\exp\left(-\beta l\right)\right>^{4k}}\right]\right)^{1/2}\right)\\
\leq& \frac{K}{N^{k}},
\end{split}
\end{align}
\begin{align}\label{prop1:eq6}
\begin{split}
E\left[\frac{B'-B}{B'}\right]^{4k}
\leq &E\left[\frac{B(1)e^{h}+B(-1)e^{-h}}{\exp(\beta r)e^h+\exp(-\beta r)e^{-h}}\right]^{4k}\\
\leq &2^{4k}\left(E\left[\frac{B(1)}{\exp(\beta r)}\right]^{4k}
+E\left[\frac{B(-1)}{\exp(-\beta r)}\right]^{4k}\right)\\
\leq &2^{4k}\left(\left(EB(1)^{8k}\right)^{1/2}\left(E\left[\frac{1}{\exp\left(8k\beta r\right)}\right]\right)^{1/2}\right.\\
&\left.+
\left(EB(-1)^{8k}\right)^{1/2}\left(E\left[\frac{1}{\exp\left(-8k\beta r\right)}\right]\right)^{1/2}\right)\\
\leq& \frac{K}{N^{2k}},
\end{split}
\end{align}
and
\begin{align}\label{prop1:eq5}
\begin{split}
E\left[\frac{A'}{B}\right]^{4k}
=&E\left[\frac{E_\xi\left[U(\xi+r)\left(\exp(\beta(\xi+r))e^h+\exp(-\beta(\xi+r))e^{-h}\right)\right]}{
\left<\exp(\beta l)e^h+\exp(-\beta l)e^{-h}\right>}\right]^{4k}\\
\leq &2^{4k}\left(E\left[U(\xi+r)\frac{\exp(\beta(\xi+r))}{\left<\exp(\beta l)\right>}\right]^{4k}+E\left[
U(\xi+r)\frac{\exp(-\beta(\xi+r))}{\left<\exp(-\beta l)\right>}\right]^{4k}\right)\\
\leq &K.
\end{split}
\end{align}
Finally, by applying $(\ref{prop1:eq4}),$ $(\ref{prop1:eq6})$, and $(\ref{prop1:eq5})$ to the
following inequality
$$
\left|\frac{A}{B}-\frac{A'}{B'}\right|^{2k}\leq 2^{2k}\left(\left|\frac{A-A'}{B}\right|^{2k}+\left|\frac{A'}{B}\right|^{2k}
\left|\frac{B'-B}{B'}\right|^{2k}\right),
$$
and using H\"{o}lder's inequality, the announced result follows.

\end{proof}


Recall that $q$ is defined by $(\ref{Intro:eq2}).$ We also define $q_-$ as the unique solution of
$q_-=E\tanh^2(\beta_-z\sqrt{q_-}+h),$ where $\beta_-=\sqrt{(N-1)/N}\beta$ and $z$ is a standard Gaussian distribution.
Notice that the existence and uniqueness of $q_-$ are always guaranteed since we only consider the high temperature region, that is, $\beta_-<1/2.$
Also, recall the quantity $\gamma_i$ from $(\ref{thm2:prep:eq1}).$

\begin{lemma}\label{lem5}
There is a constant $L>0$ so that
\begin{align}\label{lem5:eq1}
|q-q_-|\leq \frac{L}{N}
\end{align} for every $\beta<1/2$, $h$, and $N.$ Let $\beta_0<1/2$ be fixed. Then for every $1\leq i\leq N$, $\beta\leq \beta_0,$ and $h,$
\begin{align}\label{lem5:eq2}
E\left[\gamma_i
-\frac{1}{\sqrt{N-1}}\sum_{j\leq N,j\neq i}g_{ij}\left<\sigma_j\right>_-\right]^{2k}\leq \frac{K}{N^{k}}
\end{align}
where $K$ is a constant depending only on $\beta_0$ and $k$.
\end{lemma}

\begin{proof}
The inequality $(\ref{lem5:eq1})$ is from Lemma 1.7.5. \cite{Talag10}, while $(\ref{lem5:eq2})$ follows from
the inequalities on page $86$ of \cite{Talag10}.

\end{proof}


\begin{Proof of theorem} $\mathbf{\ref{thm2}}:$
By symmetry among the sites, it suffices to prove $(\ref{thm2:eq1})$
is true when $i=N.$ Recall from $(\ref{Intro:eq7})$ and $(\ref{thm2:prep:eq1}),$ for simplicity, we set $l=l_N$, $\gamma=\gamma_N.$ We also set
$$
l_-=\frac{1}{\sqrt{N-1}}\sum_{j\leq N-1}g_{Nj}\sigma_j,\quad r_-=\left<l_-\right>_-,
$$
where $\left<\cdot\right>_-$ is the Gibbs measure with Hamiltonian $(\ref{Intro:eq6})$ and inverse temperature $\beta_-=\sqrt{(N-1)/N}\beta.$
From Proposition $\ref{prop1},$ we know
\begin{equation}\label{thm2:pf:eq1}
E\left[\frac{\left<U(l_-)\cosh(\beta_- l_-+h)\right>_-}{\left<\cosh(\beta_- l_-+h)\right>_-}-\frac{E_\xi\left[U(\xi+r_-)
\cosh(\beta_-(\xi+r_-)+h)\right]}{\exp\left(\frac{\beta_-^2}{2}(1-q_-)\right)\cosh(\beta_- r_-+h)}\right]^{2k}\leq \frac{K}{N^k},
\end{equation}
where $K$ is a constant depending on $\beta_0,k,$ and $U$ only.
The goal of the proof is then to prove that $(\ref{thm2:pf:eq1})$ is very close to $(\ref{thm2:eq1}).$ We divide our estimates into several steps.

{\bf Step 1.}
Similar to $(\ref{Intro:eq5})$, from the Gibbs measure, the following identity holds
$$
\left<U(l)\right>=\frac{\left<U(l)\cosh(\beta l+h)\right>_-}{\left<\cosh(\beta l+h)\right>_-}.
$$
Note that $\beta_-l_-=\beta l.$ Therefore,
$$
\frac{\left<U(l_-)\cosh(\beta_- l_-+h)\right>_-}{\left<\cosh(\beta_- l_-+h)\right>_-}=
\frac{\left<U(l_-)\cosh(\beta l+h)\right>_-}{\left<\cosh(\beta l+h)\right>_-},
$$
and this quantity is very close to $\left<U(l)\right>$
in the sense that
\begin{align}\label{thm2:pf:eq2}
\begin{split}
&E\left[\frac{\left<U(l_-)\cosh(\beta_- l_-+h)\right>_-}{\left<\cosh(\beta_- l_-+h)\right>_-}-\left<U(l)\right>
\right]^{2k}\\
=&E\left[\frac{\left<\left(U(l_-)-U(l)\right)\cosh(\beta l+h)\right>_-}{\left<\cosh(\beta l+h)\right>_-}\right]^{2k}\\
\leq &E\left[U(l_-)-U(l)\right]^{2k}\\
\leq &\frac{K}{N^k}.
\end{split}
\end{align}
Indeed, the first inequality is true by the use of Jensen's inequality. The second inequality holds by using
the mean value theorem and $\sqrt{N}/\sqrt{N-1}-1\leq 2/\sqrt{N}$ together to obtain
\begin{align*}
|U(l_-)-U(l)|\leq &|l_--l|\sup_{0\leq t\leq 1}|U'(tl_-+(1-t)l)|\\
\leq& \frac{\sqrt{2}}{\sqrt{N}}|l|\sup_{0\leq t\leq 1}|U'(tl_-+(1-t)l)|
\end{align*}
and then applying Lemma $\ref{lem3}.$

{\bf Step 2.} Since for any $a,b,c\in\Bbb{R},$
\begin{align}\label{thm2:pf:eq3}
\frac{\cosh(c(\xi+a)+h)}{\cosh(ca+h)}=&\frac{e^{c\xi}}{1+e^{-2ca-2h}}+\frac{e^{-c\xi}}{1+e^{2ca+2h}}
\end{align}
and
\begin{align*}
\frac{d}{da}\left(\frac{1}{1+e^{\pm 2(ca+h)}}\right)=\frac{\mp 2c}{(e^{ca+h}+e^{-ca-h})^2},
\end{align*}
we have
$$
\left|\frac{\cosh(c(\xi+a)+h)}{\cosh(ca+h)}-\frac{\cosh(c(\xi+b)+h)}{\cosh(cb+h)}\right|\leq |c||a-b|\cosh (c\xi).
$$
Therefore, by H\"{o}lder's inequality, Lemmas $\ref{lem3}$ and $\ref{lem5},$
\begin{align}\label{thm2:pf:eq4}
\begin{split}
&E\left[E_\xi\left[U(\xi+r_-)\left(\frac{\cosh(\beta_-(\xi+r_-)+h)}{\cosh\left(\beta_-r_-+h\right)}-\frac{
\cosh(\beta_-(\xi+\gamma)+h)}{\cosh\left(\beta_-\gamma+h\right)}\right)\right]\right]^{2k}\\
\leq &\beta_{-}^{2k}E\left[|U(\xi+r_-)|\cosh(\beta_-\xi)|r_--\gamma|\right]^{2k}\\
\leq &\beta_{-}^{2k}\left(E\left[|U(\xi+r_-)|\cosh(\beta_-\xi)\right]^{4k}\right)^{1/2}
\left(E\left[|r_--\gamma|\right]^{4k}\right)^{1/2}\\
\leq &\frac{K}{N^{k}}.
\end{split}
\end{align}

{\bf Step 3.} Similar to the first step, by the mean value theorem, Lemmas $\ref{lem3}$ and $\ref{lem5}$, we have
\begin{align}\label{thm2:pf:eq5}
E\left[E_\xi\left[\left(U(\xi+r_-)-U(\xi+\gamma)\right)\frac{
\cosh(\beta_-(\xi+\gamma)+h)}{\cosh\left(\beta_-\gamma+h\right)}\right]\right]^{2k}\leq \frac{K}{N^k}.
\end{align}

{\bf Step 4.} Let us apply the same trick as the proof for Proposition $\ref{prop1}$ to obtain
\begin{align*}
&E\left[U(\xi+\gamma)\frac{\cosh(\beta_-(\xi+\gamma)+h)}{\cosh(\beta_-\gamma+h)}\right]^{2k}\\
\leq &E\left[U(\xi+\gamma)^{2k}\left(\frac{\exp(\beta_-(\xi+\gamma))e^h+\exp(-\beta_-(\xi+\gamma))e^{-h}}{\exp(\beta_-\gamma)e^h+\exp(-\beta_-\gamma)e^{-h}}
\right)^{2k}\right]\\
\leq &2^{2k}\left(E\left[U(\xi+\gamma)^{2k}\left(\exp(2k\beta_-\xi)+\exp(-2k\beta_-\xi)\right)\right]\right)\\
\leq &K.
\end{align*}
On the other hand, a straightforward computation gives
\begin{align*}
&\frac{\partial^2}{\partial (\beta^2)\partial q}\exp\left(-\frac{\beta^2}{2}(1-q)\right)\\
=&\frac{1}{2}\exp\left(-\frac{\beta^2}{2}(1-q)\right)
-\frac{(1-q)\beta^2}{4}\exp\left(-\frac{\beta^2}{2}(1-q)\right).
\end{align*}
Thus, from Lemma $\ref{lem5},$
\begin{align*}
&\left|\exp\left(-\frac{\beta_-^2}{2}(1-q_-)\right)-\exp\left(-\frac{\beta^2}{2}(1-q)\right)\right|\\
\leq &\left(\frac{1}{2}+\frac{1}{4}\beta^2\right)
|\beta^2-\beta_-^2||q-q_-|\\
\leq &\frac{K}{N^2}
\end{align*}
and so by Jensen's inequality
\begin{align}\label{thm2:pf:eq6}
\begin{split}
&E\left[\left(\exp\left(-\frac{\beta_-^2}{2}(1-q_-)\right)-\exp\left(-\frac{\beta^2}{2}(1-q)\right)\right)\right.\\
&\qquad\qquad\qquad\qquad\times\left.E_\xi\left[U(\xi+\gamma)\frac{
\cosh(\beta_-(\xi+\gamma)+h)}{\cosh\left(\beta_- \gamma+h\right)}\right]\right]^{2k}
\leq \frac{K}{N^{k}}.
\end{split}
\end{align}

{\bf Step 5.}
Notice that from $(\ref{thm2:pf:eq3}),$ we obtain
\begin{align*}
\frac{d}{d c}\frac{\cosh(c(\xi+a)+h)}{\cosh(ca+h)}=&\frac{\xi e^{c\xi}}{1+e^{-2(ca+h)}}+\frac{-\xi e^{-c\xi}}{1+e^{2(ca+h)}}\\
&+\frac{-2ae^{c\xi-2(ca+h)}}{(1+e^{-2(ca+h)})^2}+\frac{2ae^{-c\xi+2(ca+h)}}{(1+e^{2(ca+h)})^2}.
\end{align*}
Since
\begin{align*}
\left|\frac{2ae^{\pm c\xi\mp 2(ca+h)}}{(1+e^{\mp2(ca+h)})^2}\right|=\left|\frac{e^{\mp 2(ca+h)}}{1+e^{\mp 2(ca+h)}}\right|\left|
\frac{2ae^{\pm c\xi}}{1+e^{\mp 2(ca+h)}}\right|\leq 2|a|e^{\pm c\xi},
\end{align*}
it follows that
$$
\left|\frac{d}{d c}\frac{\cosh(c(\xi+a)+h)}{\cosh(ca+h)}\right|\leq 2\left(|\xi|+2|a|\right)\cosh(c\xi)
$$
and for $c'<c,$
\begin{align}\label{thm2:pf:eq9}
\left|\frac{\cosh(c(\xi+a)+h)}{\cosh(ca+h)}-\frac{\cosh(c'(\xi+a)+h)}{\cosh(c'a+h)}\right|\leq 2\left(|\xi|+2|a|\right)\int_{c'}^c
\cosh(t\xi)dt.
\end{align}
Since from Lemma $\ref{lem3},$ we know
\begin{align*}
E\left[|U(\xi+\gamma)|(|\xi|+2|\gamma|)\sup_{0\leq t\leq \beta_0}\cosh(t\xi)\right]^{2k}\leq K,
\end{align*}
it follows, by $(\ref{thm2:pf:eq9})$ and Jensen's inequality, that
\begin{align}\label{thm2:pf:eq8}
E\left[E_\xi\left[U(\xi+\gamma)\left(\frac{
\cosh(\beta_-(\xi+\gamma)+h)}{\cosh\left(\beta_-\gamma+h\right)}-\frac{
\cosh(\beta(\xi+\gamma)+h)}{\cosh\left(\beta \gamma+h\right)}\right)\right]\right]^{2k}
\leq \frac{K}{N^{k}}.
\end{align}

{\bf Step 6.}
Combining $(\ref{thm2:pf:eq1})$, $(\ref{thm2:pf:eq2})$, $(\ref{thm2:pf:eq4})$, $(\ref{thm2:pf:eq5})$, $(\ref{thm2:pf:eq6}),$ and $(\ref{thm2:pf:eq8}),$
we finally obtain
\begin{align}\label{thm2:pf:eq7}
E\left[\left<U(l)\right>-\frac{E_\xi\left[U(\xi+\gamma)
\cosh(\beta(\xi+\gamma)+h)\right]}{\exp\left(\frac{\beta^2}{2}(1-q)\right)\cosh(\beta \gamma+h)}\right]^{2k}\leq \frac{K}{N^k}.
\end{align}
Substitute the identity
\begin{align*}
\frac{(x-\gamma)^2}{2(1-q)}\mp(\beta x+h)=&\frac{1}{2(1-q)}(x-(\gamma\pm(1-q)\beta))^2\mp(\beta \gamma+h)-\frac{\beta^2}{2}(1-q)
\end{align*}
in the right-hand side of
\begin{align*}
&E_\xi\left[U(\xi+\gamma)\cosh(\beta(\xi+\gamma)+h)\right]\\
=&\int \frac{U(x)}{\sqrt{2\pi(1-q)}}\frac{e^{\beta x+h}+e^{-\beta x-h}}{2}
\exp\left(-\frac{(x-\gamma)^2}{2(1-q)}\right),
\end{align*}
then $(\ref{thm2:eq1})$ holds by $(\ref{thm2:pf:eq7})$ and we are done.
\end{Proof of theorem}
\newline

\begin{Acknowledgements}
The author would like to thank professor M. Talagrand for sharing his interesting book \cite{Talag10}, which motivated
our study of this paper.
\end{Acknowledgements}



\end{document}